\documentclass[12pt]{amsart}
\usepackage{amsmath}
\usepackage{amssymb}
\usepackage{amsfonts}
\usepackage{amsthm}
\usepackage{verbatim}
\usepackage{amscd}
\usepackage{cite}
\usepackage{leftidx}
\usepackage{enumerate}
\usepackage{txfonts}
\usepackage{manfnt}
\usepackage{amscd}
\usepackage[mathscr]{eucal}
\usepackage{hyperref}
\usepackage{datetime2}
\usepackage{mathpazo}
\def\cP{\mathcal P}
\def\cX{\mathcal X}
\def\cZ{\mathcal Z}
\def\cC{\mathcal C}

\def\gG{\mathbb G}
\def\fF{\mathbb F}
\def\vV{\mathbb V}

\newtheorem{thm}{Theorem} 

\newtheorem*{thm*}{Theorem}
\newtheorem*{prop*}{Proposition}
\newtheorem{cor}[thm]{Corollary}
\newtheorem*{cor*}{Corollary}

\newtheorem{lem}[thm]{Lemma}
\newtheorem*{lem*}{Lemma}

\newtheorem*{claim*}{Claim}
\newtheorem{prop}[thm]{Proposition}

\theoremstyle{remark}

\newtheorem{rem}[thm]{Remark}
\newtheorem*{rem*}{Remark}
\newtheorem{crit-rem}[thm]{Critical remark}

\newtheorem{example}[thm]{Example}
\newtheorem*{example*}{Example}

\newtheorem*{defn*}{Definition}

\def\inv{^{-1}}

\DeclareMathOperator{\coker}{coker}

\def\e{\mathrm {e}}

\def\refp #1.{(\ref{#1})}
\newcommand\carets [1]{\langle #1 \rangle}

\newcommand{\A}{\mathcal{A}}

\newcommand{\kk}{\mathbf{k}}

\newcommand{\ul}[1]{\underline {#1}}

\def\sbr #1.{^{[#1]}}
\def\sfl #1.{^{\lfloor #1\rfloor}}

\def\inv{^{-1}}
\def\?{{\bf{??}}}

\def\A{\Bbb A}

\def\C{\mathbb C}
\def\P{\mathbb P}
\def\N{\mathbb N}

\def\Z{\mathbb Z}

\def\Spec{\text{\rm Spec} }

\def\O{\mathcal O}

\def\Sym{\textrm{Sym}}

\def\rk{\text{rk}}

\def\g{\mathfrak g}

\def\1/2{\frac{1}{2}}

\def\I{\mathcal{ I}}

\def\2{{[2]}}
\def\l{\ell}
\def\nl{\newline}

\def\<{\langle}
\def\>{\rangle}

\def\2{{[2]}}
\def\l{\ell}

\def\scl #1.{^{\lceil#1\rceil}}
\def\spr #1.{^{(#1)}}
\def\sbc #1.{^{\{#1\}}}

\def\subpr#1.{_{(#1)}}

\def\beq{\begin{equation*}}
\def\eeq{\end{equation*}}

\def\g3{{\Gamma\spr 3.}}

\newcommand{\eqspl}[2]{
\begin{equation}\label{#1}
\begin{split}
#2\end{split}\end{equation}}
\newcommand{\eqsp}[1]{\begin{equation*}
\begin{split}#1\end{split}\end{equation*}}

\newcommand{\exseq}[3]{
0\to #1\to #2\to #3\to 0
}

\newcommand{\beginalphaenum}{
\begin{enumerate}\renewcommand{\labelenumi}{ }
\item \begin{enumerate}
}

\def\eex{\end{rm}\end{example}}

\begin{document} 

\title{Interpolation of rational scrolls}
\author 
{Ziv Ran}
\thanks{arxiv.org/math.ag/ }
\date{\DTMnow}


\address {\nl UC Math Dept. \nl
	Skye Surge Facility, Aberdeen-Inverness Road
	\nl
	Riverside CA 92521 US\nl 
	ziv.ran @  ucr.edu\nl
	\url{https://profiles.ucr.edu/app/home/profile/zivran}
}

\subjclass[2010]{14n25, 14j45, 14m22}

\begin{abstract}
	
	We show in many cases that there exist rational scrolls which are balanced,
	i.e. they contain the
	expected number of general linear spaces as rulings. For example, there exist balanced scrolls
	of degree $mk+1$ and fibre dimension $k$ in $\P^{2k+1}$ for all $m\geq 1$.
\end{abstract}
\maketitle
As is well known from classical projective geometry, 
3 general lines in $\P^3$ lie on a unique quadric surface. Extrapolating from there,
one can pose the following \emph{rational scroll interpolation problem}. 
Define a \emph{rational $(k,1)$-scroll} in $\P^n$ to be the union of
a family of linear $\P^k$s parametrized by a rational curve. 
Then given a general collection of $q$ linear $\P^k$s in $\P^n$, what is the smallest degree
of a rational  $(k,1)$-scroll containing all these as rulings? There is an obvious \emph{expected}
answer to this question, namely the smallest such degree $e$ is
\[ e=\lceil\frac{ (q-1)((k+1)(n-k)-1)+2}{n+1}\rceil.\]
This expectation is based on  the fact that a rational $(k,1)$-scroll of degree $e$ 
corresponds to a rational curve $C$
of degree $e$ in the Grassmannian $\gG=\gG(k,n)$ and if the scroll is to contain $q$ general rulings then
the normal bundle $N=N_{C/\gG}$ must satisfy 
\eqspl{euler}{\chi(N(-qP))\geq 0} where $P\in C.$
Then the question becomes whether the actual degree equals the expected degree.
The case where $k=0$, i.e. $\gG(k,n)=\P^n$ is treated in \cite{ran-normal}, 
while the analogous question for Fano hypesurfaces in $\P^n$ is treated in
\cite {coskun-riedl}, \cite{shen-normal}, \cite{chen-zhu},
\cite{hypersurf} and \cite{caudatenormal} .\par
Now a rational curve $C$ of degree $e$ in $\gG(k,n)$ is said to be \emph{balanced}
if 
\eqsp{&\chi(N(-qP))=h^0(N(-qP)),\\  &q:=[(e(n+1)-2)/((k+1)(n+1)-1)]=[\deg(N)/\rk(N)],}
i.e. if the scroll corresponding to 
 $C$ contains the maximum (=expected) number of general $\P^k$s as rulings.
The purpose of this paper is to show that for many $k, n$, $\gG(k,n)$
contains infinitely many families of balanced rational curves. For example, when $k=1$
or $n=2k+1$,
the set of degrees $e$ of balanced rational curves contains some arithmetic progressions.
See Theorems \ref{k=1}, \ref{any k}, \ref{mirror}, Corollary \ref{oddG} 
as well as Examples \ref{odd-sec}, \ref{4n_1},
\ref{4n_1+2} , \ref{r=0} and \ref{alphabeta-example} for more precise statements. \par
The proof is based on a 'balanced plus perfect' strategy. We use a 'deformation to the normal bundle' 
type construction
to degenerate a Grassmannian $\gG$ to a reducible variety $\fF_1\cup \fF_2$ where $\fF_1$
is the blowup of $\gG$ is a Schubert cycle $\Sigma$ and $\fF_2$ is essentially
the projectivization $\P(N_{\Sigma/\gG}\oplus \O)$. Then we construct genus-0 
nodal curves $C_1\cup C_2$
with $C_i\subset\fF_i$ balanced, $i=1,2$ and at least one of them perfect (having
normal bundle that is a twist of a trivial bundle). Such a curve deforms to a balanced rational curve in
$\gG$.\par
The Schubert cycles used here are rather special (e.g. sub-Grassmannians). Hopefully
more general Schubert cycles can lead to more balanced scrolls.

\section*{Conventions and preliminaries}
We work over an algebraically closed field $\kk$ and use the EGA convention for projective
and Grassmannian spaces and bundles. Thus 
\[\gG=\gG(k,n)=G(k+1, n+1)\] denotes the space
of $(k+1)$- dimensional quotients or codimension-$(k+1)$ subspaces 
 $V^{n-k}\subset\kk^{n+1}$ with its tautological rank-$(k+1)$ quotient
 bundle $Q$ and tautological rank-$(n-k)$ subbundle $S$.
 By duality \[\gG=G(n-k, n+1)\] so whenever convenient we may also consider $\gG$
 as parametrizing  
 $(k+1)$- dimensional  or codimension-$(n-k)$ subspaces 
 $V^{k+1}\subset\kk^{n+1}$
 Thus $\gG$  carries two
 sub- and two quotient bundles, to be denoted $S^{n-k}, S^{k+1}$,
 $Q^{k+1}, Q^{n-k}$ if needed to avoid ambiguity.
 These are related by $(S^{n-k})^*=Q^{n-k}, (S^{k+1})^*=Q^{k+1}$.\par
 \subsection{Balanced bundles} 
 A bundle $A$ on $\P^1$ is \emph{balanced} if $H^1(A\otimes A^*)=0$; if moreover
 $\rk(A)|\deg(A)$, so $A$ is a twist of a trivial bundle, $A$ is
 \emph{perfectly balanced}. The tensor product of a balanced bundle and a perfectly
 balanced one is balanced, and slopes add.
 A rational curve $C\to X$ is (perfectly)
 balanced if its normal bundle is. 
 The slope of $C$ is by definition that of its normal bundle.
 If $X$ is a Grassmannian $C$
 is said to be quot- or sub- balanced if the appropriate tautological bundle
 is balanced. This just means that the respective (mutually dual)  projective scrolls
 corresponding to $C$ are balanced, i.e. projectivizations of balanced bundles.\par
\subsection{Balanced extensions}
Given an exact sequence of vector bundles on $\P^1$
\[\exseq{E_1}{E}{E_2}\]
	with respective slopes $s_1, s, s_2$ such that $E_1$ and $E_2$
	are balanced and the matching condition
	$[s_1]=[s_2]$ holds, then $E$ is balanced and $[s]=[s_1]$
		(see \cite{caudatenormal}, Lemma 25).\par 
 If $X$ is fibred over $B$, the normal bundle $N_{C/X}$ fits in an exact sequence
 \[\exseq{N_v}{N_{C/X}}{N_{C/B}}\]
 	where $N_v=T_{X/B}|_C$ and $N_h=N_{C/B}$ are the vertical and horizontal 
 	normal bundles . The
 	respective slopes of $N_v$ and $N_h$ are called the vertical and horizontal slopes
 	of $C$ and denoted $s_v$ and $s_h$.
 	If $N_v$ and $N_h$ are balanced and their
 	respective slopes satisfy $[s_v]=[s_h]$, then $N_{X/X}$ is balanced.
 
 	\subsection{Balanced kernels}
 	We will also use the following result
 	which is essentially in  \cite{caudatenormal}, Lemma 26 (with trivial enahncements):
 	\begin{lem}\label{balancing-lem}
 		Let $E$ be a rank- $r$ bundle on $\P^1$ of the form
 			\eqspl{e}
 		{E=r_0\O(a)\oplus r_1\O(a-1)\oplus...\oplus r_b\O(a-b),\ \ r_0>0, b\geq 0,  r_1,...,r_b\geq0}
 		 Then for $p_1,...,p_s\in\P^1$ distinct,
 		  the kernel $E_s$ of a sufficiently general map to a skyscraper sheaf
 		$E\to\bigoplus\limits_{i=1}^{s}\kk_{p_i}$ is balanced (resp. perfectly balanced) if 
 		$s\geq \sum r_i(b-i)$
 		(resp. $s=\sum r_i(b-i)$).
 		In particular, if $E$ is balanced with upper rank $r^+$
 		then the kernel of a sufficiently general map $E\to\bigoplus\limits_{i=1}^{s}\kk_{p_i}$
 		is balanced for any $s$ and  perfectly balanced  if $s$ is a multiple of $r^+$.
 		\end{lem}
 	Let $r_{\min}(E), \mu_{\min}(E)$ denote respectively the rank and slope of the maximal rank, minimal
 	slope quotient of $E$ (i.e. in the above notation, $r_b, a-b$).
 	Then one way to prove the Lemma is to
 	 define the 'unbalanced degree' of $E$ as above as 
 	 \[u(E)=\sum r_i(b-i)=\deg(E)-\rk(E)\mu_{\min}(E)\]
 	Note $u(E)\geq \rk(E)-r_{\min}(E)$ with equality iff $E$ is balanced. 
 	Then, taking $s=1$, one can easily prove that, unless $E$ is perfectly balanced,
 	we have $u(E_1)=u(E)-1$. Since $u(E)\geq 0$, repeating
 	the modification $E\mapsto E_1$ eventually leads to a balanced and then perfectly balanced
 	bundle.
 	\begin{cor}\label{kernel}
 		Notations as above, set $\mu_{\min}(E)=a-b$, i.e. the smallest slope of a quotient of $E$.
 		Then a sufficiently general map $E\to\O(\l)$ has balanced kernel provided
 		\[\l\geq u(E)+\mu_{\min}(E)=\deg(E)-(\rk(E)-1)\mu_{\min}(E).\]
 		\end{cor}
 	\begin{proof} Pick a summand $\O(a)\subset E$ and a general map $\O(a)\to\O(\l)$. Then
 		applying Lemma \ref{balancing-lem} to a general map $E/\O(a)\to\O(\l)/O(a)$ 
 		yields that the kernel is balanced provided
 		\[\l-a\geq u(E/\O(a))=u(E)-b,\]
 		i.e. $\l\geq u(E)+a-b=u(E)+\mu_{\min}(E)$.
 		\end{proof}
 	\begin{lem}\label{to-torsion}
 		Let $E$ be a bundle on $\P^1$ and $\tau$ a torsion sheaf of
 		the form $\bigoplus\limits_{i=1}^m \kk_{p_i}$ with $p_i\in\P^1$ distinct. Then the kernel
 		$K$  of
 		a general map $E\to \tau$ is either balanced or $u(K)=u(E)-m$.
 		\end{lem}
 	\begin{proof} An easy induction on the rank of $E$.
 	\end{proof}
 	\begin{cor}
 		Let $E$ be a balanced bundle on $\P^1$ and $a\leq b\in\Z$.
 		Then the kernel $K$ of a general surjection $E\oplus\O(a)\to\O(b)$ is 
 		either balanced or has $u(K)=u(E)-(b-a)$.
 		\end{cor}
 	\begin{proof}
 		Apply the Lemma to the induced map $E\to\coker(\O(a)\to\O(b))$.
 		\end{proof}
 	\begin{lem}\label{balanced-plus-trivial}
 		Let $E$ be a globally generated balanced bundle of slope $s$
 		 on $\P^1$. Then the kernel of a general surjection $E^*\oplus a\O\to\O(b)$
 		 is balanced provided $b\geq (a-1)s$.
 		\end{lem}
 	\begin{proof}
 		Similar to the above, by induction on $a$.
 		\end{proof}
 	\subsection{Rational curves in Grassmannians} A curve $C\to\gG(k, n)$
 	corresponds to a surjection on $C$, $\phi: (n+1)\O\to E$ where $E$ is a rank-$(k+1)$ bundle, i.e.
 	$C$ corresponds 
 	to an $(n+1)$-dimensional subspace of $H^0(E)$ generating $E$ where $E$ is a bundle 
 	of rank $(k+1)$. Then $E$ may be identified with the restriction of the tautological
 	quotient bundle $Q^{k+1}_{\gG}$ on $C$. When $C\simeq \P^1$, $E$ is the $h^0$-preserving limit of a balanced
 	bundle, and it follows that for $C$ general, i.e. general in its family,
 	 the restriction of $Q^{k+1}_\gG$ is balanced.
 	Applying this reasoning to the 'transpose' $\phi^*: (n+1)\O\to (\ker(\phi))^*$
shows that 	the tautological subbundle $S^{n-k}_\gG|_C=\ker(\phi)$ is balanced as well
for $C$ general. 
\subsection{Schubert cycles, flag resolutions}\label{schubert}
Here we think of $\gG(k,n)$ as parametrizing $(k+1)$-dimensional vector
subspaces of $\kk^{n+1}$.
Fix a flag $\P^0\subset\P^1\subset...\subset \P^n$ and
the corresponding vector space flag $V^1\subset...\subset V^{n+1}=V$.
 Let $(a.)$ be a partition
of the form
\eqspl{partition}{n-k\geq a_0\geq...\geq a_k\geq 0,  a_{k+1}:= -1.}
The associated Schubert cycle of codimension $A=\sum a_i$ in $\gG(k,n)$ is defined as
\eqspl{schubert-eq}
{\Sigma(a.)=\{L: \dim(L\cap\P^{n-k-a_i+i})\geq i, i=0,...,k\}.}
(This is actually contained in $\gG(k, n-a_k)$.)
Schubert cycles are often singular where one of the inequalities in \eqref{schubert-eq}
are strict, and can be resolved by considering suitable flags.

We define the \emph{strict flag resolution} 
\[\tilde\Sigma(a.)\to\Sigma(a.)\] as follows.
First define integers $k_0<...<k_u=k$ inductively as follows. 
\[k_0=\min(i:a_{i+1}<a_i), k_{j+1}=\min(i>k_j:a_{i+1}<a_i).\]
These are the 'break points' of the partition $(a.)$. Thus if the partition $(a.)$
is without repetitions then $k_i=i, i=0,...,k$. Note that $\Sigma(a.)$ depends only
on $a_{k_0}>...>a_{k_u}$:
\eqspl{sigma*}{
\Sigma(a.)=\{L: \dim(L\cap\P^{n-k-a_{k_i}+k_i})\geq k_i, i=0,...,u\}	
}
and $A=\sum (k_i-k_{i-1})a_{k_i}$.
Set \[\Sigma_0=\gG(k_0, n-k+k_0-a_{k_0})=G(k_0+1, n-k+k_0-a_{k_0}+1).\]
This is endowed with a tautological subbundle $S^0_{k_0+1}$ and quotient bundle
$Q^0_{n-k-a_0}$. Note that $\Sigma_0$ is a point iff $a_{k_0}=n-k$.\par
Then let \[F_1=V^{n-k+1+k_1-a_{k_1}}\otimes\O/S^0_{k_0+1}=
Q^0_{n-k-a_0}\oplus (V^{n-k+1+k_1-a_{k_1}}/V^{n-k+1+k_0-a_{k_0}})\otimes\O,\]
\[\Sigma_1=G(k_1-k_0, F_1).\]
This is endowed with tautological bundle $S^1_{k_1-k_0}, Q^1_{n-k-a_{k_1}}$
	as well as a larger 'cumulative' tautological subbundle 
	$T^1_{k_1+1}\subset V^{n-k+1+k_1-a_{k_1}}\otimes\O$ which is filtered with graded
	$S^0\oplus S^1$. \par
In general if $F_i, \Sigma_i$ are defined, and on $\Sigma_i$
we have the quotient bundle $Q^i_{n-k-a_{k_i}}$ and subbundle $T^i_{k_i+1}$, then let
\[F_{i+1}=V^{n-k+1+k_{i+1}-a_{k_{i+1}}}\otimes\O / T^i= Q^i_{n-k-a_{k_i}}\oplus 
(V^{n-k+1+k_{i+1}-a_{k_{i+1}}}/V^{n-k+1-k_i+a_{k_i}})\otimes\O,\]
\[ \Sigma_{i+1}=G(k_{i+1}-k_i, F_{i+1}),\]
and $S^{i+1}_{k_{i+1}-k_i}, Q^{i+1}_{n-k-a_{k_{i+1}}}, T^{i+1}_{k_{i+1}}$ be as before.
Finally
\[\tilde\Sigma(a.)=\Sigma_{u}.\]
Then $\tilde\Sigma(a.)$ parametrizes partial flags $(L_{k_0}\subset...\subset L_{k_u}= L_k)$
of the indicated dimensions such that $L_{k_i}\subset \P^{n-k-a_{k_i}+k_i}, i=0,...,u$.
\par
Obviously $\tilde\Sigma(a.)$ is smooth, as are each of the morphisms
$\Sigma_{i+1}\to\Sigma_i$.
Note the surjection $(n+1)\O\to Q^u_{n-k}$ with kernel $T^{u}_{k+1}$ whose graded is
$\mathrm{gr}_.T^u=\bigoplus\limits_{i=0}^u S^{i}$. This yields a map
$\tilde\Sigma(a.)\to \gG(k,n)$ which is birational onto $\Sigma(a.)\subset\gG(k, n-a_k)$. \par
\begin{lem}\label{small}
	$\tilde\Sigma(a.)\to \Sigma(a.)$ is a small resolution.
	\end{lem}
\begin{proof}
Because $a_{k_i}-a_{k_{i+1}}>0$, the exceptional locus of the natural map
$\Sigma_i\to\Sigma_*(a.)$, i.e. the degeneracy locus of the map 
$S_{k_{i+1}-k_i}\to (V^{n-k+1+k_{i+1}-a_{k_{i+1}}}/V^{n-k+1-k_i+a_{k_i}})\otimes\O$,
has codimension $a_{k_{i+1}}-a_{k_i}+1>1$. And because $\tilde\Sigma(a.)\to\Sigma_i$
is flat, the same is true for the inverse image of the latter exceptional locus is
$\tilde\Sigma(a.)$.
\end{proof}
 The latter is the reason for working with
partial flags as above rather than the more standard full flags.\par

Two simpler cases of this construction are the extremal ones:\par
\begin{example}
	When $(a.)$ is strictly decreasing. Then all $k_i=i$ and the Grassmannian bundles
	are projective bundles.\end{example}
\begin{example}	When $a_0=...=a_{\alpha-1}=n-k, a_\alpha,=...=a_k=\beta$, 
	$\Sigma(a.)=\gG(k-\alpha, n-\alpha-\beta)$ is smooth and equals $\tilde\Sigma(a.)$
	(concretely, this is the locus of subspace $L$ containing $\P^{\alpha-1}$ and contained in
	$\P^{n-\beta}$). Note that in this case $\Sigma_0=\gG(\alpha-1, \alpha-1)$ is a point,
	$Q^0_{n-k-a_0}=0$ while $S^0_{k_0+1}$ has rank $\alpha$.
	\end{example}
/*******\par
*********/\par
\subsection{Normal bundle}\label{normal-sec}
To simplify notation, set
\[W_i=V^{n+1-k+k_i-a_{k_i}}, 0=0,...,u, W_{-1}=0.\]
Define a sheaf $G_i$ on $\tilde\Sigma(a.)$ by the exact sequence

/******
*****/

\eqspl{g}{0\to T^{u}/T^i\to (V/W_i)\otimes\O\to G_i\to 0, 0\leq i\leq u
	}
Thus $G_0$ is the pullback of the tautological quotient bundle $Q_{n-k}$
while $G_u=(V/W_u)\otimes\O$.
Note that over the big open subset $\Sigma_0(a.)\subset\Sigma(a.)$ where all
the inequalities of \eqref{sigma*} are equalities, the map
$\tilde\Sigma(a.)\to \Sigma(a.)$ is locally an iso, and $G_i$ is locally
free of rank $a_{k_i}$, and coincides with the quotient of the tautological
rank-$(n-k)$ quotient bundle $Q^u_{n-k}$ by the image of $V^{n-k-a_{k_i}+k_i+1}\otimes\O$. 
Then we have (recall $k_{-1}=-1$):
\begin{lem}\label{normal}
	The normal bundle $N_{\Sigma_*^0(a.)\to\gG}$ admits a filtration with graded
	$\bigoplus\limits_{i=0}^u S^{i*}_{k_i-k_{i-1}}\otimes G_{i}$.
	
	\end{lem}

\begin{example}
	For the 'simple' Schubert cycle $\Sigma=\gG(k-\alpha, n-\alpha-\beta)$, the 
	normal bundle is filtered with graded pieces
	$\alpha Q_{n-k-\beta}, \beta S^*_{k+1-\alpha}, \alpha\beta\O$.
	\end{example}
/***************
\subsubsection{Duality}

********************/
\subsection{Schubert supercycles and the standard degeneration}\label{super}
Let $\Sigma=\Sigma(a.)\subset\gG=\gG(k,n)$ be a Schubert cycle as above. 
The associated \emph{standard degeneration} of $\gG$ is
\[\pi:B_{\Sigma\times 0}\gG\times\A^1\to\A^1.\] 
The nonzero fibres of $\pi$ are $\gG$, while
\[\pi\inv(0)=\fF_1\cup_E\fF_2\]
where 
\[\fF_1=B_\Sigma\gG, \fF_2=\P(N_{\Sigma/\gG}\oplus\O), E=\P(N_{\Sigma/\gG}).\]
$\fF_2$ is called the \emph{Schubert supercycle} $\Sigma^+(a.)$
associated to $(a.)$. It admits a small resolution of the form
\[\tilde\Sigma^+(a.)=\P(N_{\tilde\Sigma(a.)\to \gG}\oplus\O).\]  
Over $\Sigma_*^0(a.)$, the normal bundle in question is described by \S \ref{normal-sec}.	
 \section{Minimal surface scrolls}
See \cite{eisenbud-harris-minimal}. 
A minimal curve $C_{n-1}\subset\gG(1, n)$ corresponds to a minimal surface scroll
 in $\P^n$, i.e. a nondegenerate scroll $S_{n-1}$ of minimal degree, namely $n-1$.
We have 
\eqsp{S_{n-1}=&\P_{\P^1}(2\O(\frac{n-1}{2})),  n\  \mathrm{odd}\\ 
 =&\P_{\P^1}(\O(n/2)\oplus\O(n/2-1)), n\  \mathrm{even}.} 
This is embedded in each case by $\O(1)$. 
 $C_{n-1}$ moves in a $(n^2-3)$-dimensional family and has
 specializations of the form $C_a\cup C_b$ for all $a+b=n-1$,
 The normal bundle $N=N_{C_{n-1}/\gG(1, n)}$
 has degree $n^2-3$ and rank $2n-3$ hence slope
 \[s_{n-1}=\frac{n^2-3}{2n-3}.\]
 By an elementary calculation, we have $[s_{n-1}]=\frac{n+1}{2}$
 for $n$ odd and $[s_{n-1}]=\frac{n}{2}$ for
 $n$ even.
 \begin{prop} (i) A general $S_{n-1}$ contains $(n+3)/2$ (resp. $n/2+1$) general lines in
 	$\P^n$ as rulings for $n\geq 3$ odd (resp. $n\geq 4$ even).\par
 	(ii) (char. = 0)  $C_{n-1}$ is balanced for all $n\geq 2$.\par
 	(iii) (char. = 0) For $n\geq 3$ odd (resp. $n\geq 4$ even),
 	the scrolls $S_{n-1}$ through $(n+3)/2$ (resp. $n/2+1$) general lines
 	fill up a variety of dimension $(n+1)/2$ (resp. fill up $\P^n$).
 	\end{prop}
 \begin{rem*}
 	Assertions (ii), (iii) are very likely true in all characteristics.
 	\end{rem*}
 \begin{proof} 
 	We first consider the case $n\geq 3$ odd. We note to begin with that (ii) follows from (i):
 	in fact, write
 	\[\deg(N)=n^2-3=\frac{n-3}{2}\frac{n+3}{2}+\frac{3(n-1)}{2}\frac{n+1}{2}.\]
 	Then (ii) means exactly that this expression corresponds to a decomposition
 	\eqspl{}{N=\frac{n-3}{2}\O(\frac{n+3}{2})\oplus \frac{3(n-1)}{2}\O(\frac{n+1}{2}).}
 	This equivalent to te assertion  that $N$ has no quotient of slope $<(n+1)/2$.
 	Since char.= 0, it suffices to prove that a general $C_{n-1}$ contains $(n+3)/2$
 	general points of $\gG(1, n)$, i.e. that (i) holds.
 	To this end we use induction, with the case $n=3$ being well known.
 	For the induction step, suppose given general lines $L_1,...,L_{(n+3)/2}$ in $\P^n$
 and let $\P^{n-2}$ be the span of $L_1,...,L_{(n-1)/2}$, $\P^3$ the span of $L_{(n+1)/2},
 L_{(n+3)/2}$ and let $M=\P^{n-2}\cap\P^3$. By induction, there is an $S_{n-3}\subset\P^{n-2}$
 containing $L_1,...,L_{(n-3)/2},M$ and an $S_2\subset\P^3$ containing
 $L_{(n+1)/2}, L_{(n+3)/2}, M$ and clearly $S_{n-3}\cup S_2$ smooths out to $S_{n-1}$.
 QED $n$ odd.\par 
 
  	Now assume $n$ is even.
 	Let us write
 	\[\deg(N)=n^2-3=(3n/2-3)(n/2+1)+(n/2)(n/2).\]
 	Then (ii) amounts to proving that this expression corresponds to a decomposition
 	\eqspl{}{N=(3n/2-3)\O_{C_{n-1}}(n/2+1)\oplus (n/2)\O_{C_{n-1}}(n/2).}
 	As above, this follows from (i).
Thus it suffices to prove that the general	$S_{n-1}$ contains $n/2+1$ general lines 
 $L_1,...,L_{n/2+1}$ as rulings.
 Let $\P^{n-1}$ be the span of $L_1,...,L_{n/2}$, let $S_{n-2}$ be general through
 $L_1,...,L_{n/2}$ and through $p:=L_{n/2+1}\cap\P^{n-1}$, let $M$ be the ruling
 of $S_{n-2}$ through $p$ and let $\P^2$ be the span of $M$ and $L_{n/2+1}$. 
 Then $S_{n-2}$ plus the pencil in $\P^2$ through $p$ is a specialization of
 $S_{n-1}$ containing $L_1,...,L_{n/2+1}$.
 \par
 Finally to prove (iii) we use induction. For $n$ odd the initial case $n=3$
 is well known. For the induction step, consider the total space $\ul S_{n-1}$
 of the family of minimal 
 scrolls through $(n+3)/3$ fixed rulings, with its natural map $f_{n-1}:\ul S_{n-1}\to\P^n$.
 Then our assertion is that $f_{n-1}$ is generically finite to its image. This can be checked
 for a reducible scroll of the form $S_{n-3}\cup S_2$ as above, at a general point of $S_{n-3}$, 
 where it holds by induction.
 \par Finally assume $n$ even and $\geq 4$. Notations as above, we want to prove that
 $f_{n-1}$ has general fibre dimension $n/2-1$. First for $n=4$, the total space $\ul S_3$
 	has a divisor corresponding to surfaces of the form $S_2\cup S_1$ and when restricted
 	on that divisor $f_3$ has fibre dimension 0 at a general point of $S_2$. 
 	Therefore overall $f_3$ has general fibre dimension 1,
 	as required. For $n\geq 6$ we argue similarly but using instead a reducible $S_{n-1}$ scroll
 	of the form $S_{n-3}\cup S_2$ to show inductively that $f_{n-1}$ has general fibre dimension $n/2-1$.
 	\end{proof}
 Note that it follows from the Proposition that
 for $n$ odd (resp. even), the deformations of $C_{n-1}$ containing
 $(n+3)/2$ (resp. $n/2+1$ fixed rulings fill up a subvariety of $\gG(1, n)$
 of dimension $(n-1)/2$ 
 (resp. $3n/2-2$).
 We actually need a slightly stronger fact:
 \begin{cor}\label{perfect-cor}(char = 0)
 	Let $[L]\in C_{n-1}$ be  general and $p\in L$ a general point. 
 	For $n\geq 3$ odd (resp. even) let $V$ be a general $\P^{(n-1)/2}$ through $p$
 	(resp. a general $\P^{n/2+1}$ containing $L$). Let $\Sigma\subset\gG(1, n)$
 	be the subvariety consisting of lines meeting $V$ (resp. containing $p$
 	and contained in $V$ and let $\tilde\gG$ be the blowup of $\gG(1, n)$ in $\Sigma$.
 	Then the birational transform of $C_{n-1}$ in $\tilde\gG$ has normal bundle
 	$(2n-3)\O((n+1)/2)$ (resp. $(2n-3)\O(n/2)$.
 	\end{cor}
 \begin{proof}
 	$n$ odd: we may assume $V$ meets the image  $f_{n-1}(\ul S_{n-1})$ transversely in a 
 	single point. Hence
 	deformations of $C_{n-1}$ preserving the $(n+3)/2$ rulings and incidence to $V$
 	correspond to a subsheaf $(2n-3)\O(-1)$ of the normal bundle which implies the assertion.\par
 	$n$ even: a general fibre of $f_{n-1}$ is a $(n/2-1)-$ dimensional subvariety of the $\P^{n-1}$ of lines
 	through a point $p$, hence 
 	$f_{n-1}(S_{n-1})$ meets $V$ transversely in just one line through $p$
 	so we can conclude as above.
 	\end{proof}

\section{Balanced surface scrolls}
Our goal is to prove
\begin{thm}\label{k=1}
	(i) Assume $n=2n_0+1$. Then $\gG(1, n)$ contains a balanced curve of degree $e$
	provided there exist $e_0, e_1, e_2$ such that $e=e_0+e_1+(n-2)e_2$ and moreover 
	$e_0\geq n_0$ and \eqref{odd-cond1}, \eqref{odd-cond2} and \eqref{odd-cond3} hold.\par
	(ii) Assume $n=4n_1$. Then $\gG(1, n)$ contains a balanced curve of degree $e$
	provided there exist $e_0, e_1$ such that there exists a balanced curve
	of degree $e_0$ in $\gG(1, n_1+1)$ and moreover $e=e_0+(n-2)e_1$ and  
 \eqref{match4} holds.\par
 (iii) Assume $n=4n_1+2$. Then $\gG(1, n)$ contains a balanced curve of degree $e$
 provided there exist $e_0, e_1, e_2$ such that $e=e_0+e_1+(n-2)e_2$ and moreover 
 $e_0\geq n_0$ and \eqref{c2-balance} holds.\par
 
	\end{thm}
\subsection{Degeneration}\label{degeneration}
The idea is to study curves in $\gG=\gG(1, n)=G(2, n+1)=G(n-1, n+1)$ by 
a standard degeneration as in \S \ref{super}
\[\mathcal G=B_{\Sigma\times 0}(\gG\times\A^1)\stackrel{\pi}{\to}\A^1\]
for a suitable Schubert cycle $\Sigma\subset\gG$. Via $\mathcal G$,  $\pi\inv(1)=\gG$ 
degenerates to $\pi\inv(0)=\fF_1\cup_E\fF_2$ with 
\[\fF_1=B_\Sigma \gG, 
E=\P(N^*_{\Sigma/\gG}), \fF_2=\P(N^*_{\Sigma/\gG}\oplus\O).\]
We then construct a curve 
\[C_1\cup C_2=\bigcup\limits_{i=1}^\l C_{1, i}\cup C_2\subset\fF_1\cup\fF_2\]
comprised of a balanced curve $C_2$ in $\fF_2$ meeting $E$ in $\l$ points together with $\l$
translates $C_{1, 1},...,C_{1, \l}$ of  a perfect curve $C_{10}$ in $\fF_1$ meeting $E$ in 1 point.
Then $C_1\cup C_2$ smooths out to a 
balanced rational curve of degree $e=\l(\deg(C_{1,0}-1)+\deg(C_2)$ in $\gG$.
In practive we will take $C_{10}$ minimal, hence of degree $n-1$, hence
$e=\l(\deg(n-2)+\deg(C_2)$

\par 
\subsection{Case $n$ odd}\label{odd-sec}
We begin with the case $n=2n_0+1$ odd. We use Corollary \ref{perfect-cor}, and accordingly
we take for $\fF_1$ the blowup of $\gG$ in the subvariety 
$\Sigma=\Sigma(n_0, 0)\subset \gG$ consisting of lines meeting $\P^{n_0}$.
$\Sigma$ is the birational image of
\[
\tilde\Sigma=\{(p, L):p\in L\cap\P^{n_0}\}\subset\P^{n_0}\times\gG.\]
Moreover the  exceptional locus of $\tilde\Sigma\to\Sigma$ is
 $\gG(1, n_0)$ which has codimension $2(n_0+1)$ in $\gG$. 
Then Corollary \ref{perfect-cor} yields that the
birational transform of $C_1$ the blowup $\fF_1$ is indeed a perfect curve.\par
Now we study $\tilde\Sigma$ to construct the curve $C_2$.
Note that via the map $\tilde\Sigma\to\P^{n_0}$, we can identify
\[\tilde\Sigma\simeq \P(Q^*_n|_{\P^{n_0}})=\P(Q^*_{n_0}\oplus(n_0+1)\O)\]
where $Q_x$ denotes the respective quotient bundles on $\P^x$.
To construct a curve in $\tilde\Sigma$ we start with a 
general, hence balanced rational curve $C_0\subset\P^{n_0}$
of degree $e_0\geq n_0$, hence slope 
\[\frac{e_0(n_0+1)-2}{n_0-1}=e_0+2\frac{e_0-1}{n_0-1}=
\frac{e_0(n+1)-4}{n-3}=e_0+4\frac{e_0-1}{n-3}.\] 
Note $C_0$ corresponds to a general inclusion over $\P^1$ $\O(-e_0)\to\O(n_0+1)\O$
whose cokernel, i.e. $Q_{n_0}|_{C_0}$ is balanced by \cite{caudatenormal}, Lemma 26.
Then we lift
$C_0$ to a curve $C_{01}\subset\tilde\Sigma$ of degree $e_1$ via a general surjection
\[\phi: Q^*_{n_0}|_{C_0}\oplus(n_0+1)\O\to\O(e_1).\]
Setting $K=\ker(\phi)$, the vertical normal bundle of $C_{01}$ is $K^*(e_1)$.
Because $Q_{n_0}$ is balanced of slope $e_0/n_0$ it follows from Corollary \ref{kernel}
that $K$ is balanced provided
\eqspl{odd-cond1}{
e_1\geq e_0.	
}
\par/************
**********/\par
Therefore $C_{01}$ will have balanced normal bundle in $\tilde\Sigma$ provided
\eqspl{odd-cond2}{
e_0+[2(e_0-1)/(n_0-1)]=e_1+[(e_0+e_1)/2n_0].	
}
Finally lifting $C_{01}$ to $C_+=C_2\to\fF'_2=\P(N^*\oplus\O))$,
where $N=(n_0+1)\O(e_0)=N_{\tilde\Sigma\to\gG}|_{C_{01}}$
amounts to a surjection
\[\phi_+:(n_0+1)\O(e_0)\oplus\O\to\O(e_+)\]
and letting $K_+=\ker(\phi_+)$, the vertical normal bundle for $C_+$ over $\tilde\Sigma$
is $K^*_+(e_+)$. Now letting $\tau_+$ denote the cokernel of the general injection
$\O\to\O(e_+)$ induced by $\phi_+$, $K_2$ is also the kernel of a general surjection
$(n_0+1)\O(e_0)\to \tau_+$, hence balanced. Clearly the slope of $K_+(e_+)$ is
$e_0+e_++e_+/(n_0+1)$. Therefore the normal bundle to $C_+\to\fF'_2$ will be balanced provided 
\eqref{odd-cond1} and \eqref{odd-cond2} hold and moreover
\eqspl{odd-cond3}{
e_0+[2(e_0-1)/(n_0-1)]=e_0+e_++[e_+/(n_0+1)].	
}
Furthermore, by choosing $C_+$ generally, it will be disjoint from the exceptional locus
of $\fF'_2\to\fF_2$, so $N_{C_+/\fF'_2}=N_{C_+/\fF_2}$, so $C_+\to\fF_2$ is balanced
as well.
Then $\bigcup\limits_{i=1}^\l C_{1i}\cup C_+$ will be balanced of degree
$e=e_0+e_1+(n-2)e_+$. This proves Theorem \ref{k=1} (i).\par
Note that taking $2n_0(n_0+1)|e_1, (n_0+1)|e_2$ we get infinitely many degrees
$e$ (in fact some arithmetic progressions) where
\eqref{odd-cond1}, \eqref{odd-cond2} and \eqref{odd-cond3} hold.\par
\subsection{Case $4|n$}\label{4n_1} Write $n=4n_1$. Set
\[\Sigma=\gG(1, n_1+1)\subset\gG.\]
Let $C_0\to\Sigma$ be a general, balanced curve of degree $e_0$. 
Then $N_{C_0/\Sigma}$ is balanced of slope $((n_1+2)e_0-2)/(2n_1-1)$.
Note that
\[N:=N_{\Sigma\to\gG}=(3n_1-1)S_2^*\]
which is balanced of slope $e_0/2$. Then a general lifting of $C_0$ to 
$C_+\to \fF_2=\P(N^*\oplus\O)$
amounts to a general surjection $N^*\oplus\O\to\O(e_+)$, i.e. 
\[N^*\to\bigoplus\limits_1^{e_+}\kk_{p_i}\]
whose kernel $K$ will be balanced of slope $-e_0/2-e_+/2(3n_1-1)$.
Then the vertical normal bundle $K^*(e_1)$ will be balanced of slope $e_0/2+e_+(1+1/(6n_1-2)$.
The slope matching matching condition takes the form
\eqspl{match4}{
e_++[\frac{e_+}{6n_1-2}]=[\frac{5e_0-4}{4n_1-2}]	
}
Choosing $e_0$ of the form $e_0=k(4n_1-2)$, then writing
\[5k-1=a(6n_1-2)+b, 0\leq b<5k-1,\]
equation \eqref{match4} holds if we take $e_+=(a-1)(6n_1-2)+b$ which yields infinitely
many values of $e=e_0+(n-2)e_+$. In fact taking for $k$ a multiple of $6n_1-2$,
we get for $e_0$ and $e_+$ an arithmetic progression of difference $(4n_1-2)(6n_1-2)$.
Therefore, provided there exists a balanced rational curve of degree $e_0$ in
$\gG(1, n_1+1)$ (see \S \ref{odd-sec}), 
there will exist a balanced rational curve of degree $e$ in $\gG$.
\subsection{Case $4|(n-2)$}\label{4n_1+2}
 Write $n=4n_1+2$, and set
\[\tilde\Sigma= \{L_0\subset L_1\}\subset\P^{n_1}\times\gG(1, n_1+2).\]
The map $\tilde\Sigma\to\gG$ is birational onto the $(2n_1+1)$- dimensional Schubert cycle
\[\Sigma=\{L_1:L_1\cap\P^{n_1}\neq\emptyset\}.\]
Via the map $\tilde\Sigma \to \P^{n_1}$,  
$\tilde\Sigma$ can be identified with the projective bundle
$\P(Q_{n_1}^*\oplus 2\O)$. 
Given a general rational curve $C_0\subset\P^{n_1}$, a general lifting $C_0$ to
$C_{01}\subset\tilde\Sigma$ amounts to a general surjection $Q^*_{n_1}\oplus 2\O\to\O(e_1)$.
By Lemma \ref{balanced-plus-trivial}, 
the kernel $K$ of this surjection will be balanced provided $e_1\geq e_0/n_1$. Moreover
The vertical normal bundle of $C_{01}$ is $K^*(e_1)$ which has slope 
$\frac{n_1+2}{n_1+1}e_1+\frac{e_0}{n_1+1}$.
Then $C_{01}\to\tilde\Sigma$ will be balanced provided
\[[\frac{n_1+2}{n_1+1}e_1]=e_0+[\frac{2e_0-2}{n_1-1}]-[\frac{e_0}{n_1+1}].\]
Now the pullback of the rank-2 tautological subbundle (or
dual of the tautological quotient bundle) to $\tilde\Sigma$
fits in an exact sequence
\[\exseq{\pi^*(\O(-1))}{S}{\O_{\tilde\Sigma}(-1)}\]
 and the normal bundle
$N=N_{\tilde\Sigma\to\gG}$ restricted on $C_{01}$ 
may be identified as 
\[N|_{C_0}=3n_1S_{C_{01}}^*\oplus\O(e_0+e_1).\] Then lifting $C_{01}$ to
$C_+\to\fF_2=\P(N_{\tilde\Sigma\to\gG}\oplus\O))$ amounts to a surjection
$(N\oplus\O)_{C_{01}}\to\O(e_+)$ or equivalently, an exact sequence on $C_{01}$
\[\exseq{K_2}{3n_1S_{C_{01}}\oplus \O}{ \bigoplus\limits_{i=1}^{e_0+e_1+e_+}k_{p_i}}.\]
The vertical normal bundle of $C_+$ then coincides with $K^*(e_+)$ and has slope
$e_++(e_++3n_1(e_1+e_0)/(6n_1+1)$ so $C_+$ will be balanced provided
\eqspl{c2-balance}{
[\frac{n_1+2}{n_1+1}e_1]+[\frac{e_0}{n_1+1}]=e_0+[\frac{2e_0-2}{n_1-1}]
=e_++[(e_++3n_1(e_1+e_+))/(6n_1+1)]	
}
This again yields balanced curves in $\gG$ of degree $e=e_0+e_1+(n-2)e_+$.

\section{Higher dimensions}
\subsection{minimal scrolls}
Let $\gG=\gG(k,n)$. As is well known, a minimal-degree curve $C_{k,n}\subset\gG$ 
corresponds to a minimal-degree
scroll $X_{k,n}\subset\P^n$, which has degree $n-k$ if $2k+1\leq n$ and $k+1$
if $2k+1\geq n$. We will assume $2k+1\leq n$.
\begin{prop}
	$C_{k,n}\subset\gG$ is balanced.
	\end{prop}
\begin{proof}
	We begin with the base case $n=2k+1$. In this case $X=X_{k,n}=\P(V\otimes\O_{\P^1}(1)), 
	V=\kk^{k+1}$. The tautological 
	exact sequence on $\gG$ restricts on $C_{k,n}$ to
	\[\exseq{V(-1)}{V\otimes H^0(\O(1))}{V(1)},\]
hence the tautological	subbundle $S_{k+1}|_{C_{k,n}}=V\otimes\O(-1)$while the tautological
		quotient bundle $Q_n|_{C_{k,n}}=V\otimes\O(1)$ so that
		$T_{\gG}|_{C_{k,n}}=V^*\otimes V\otimes\O(2)$ and the inclusion
		$T_{C_{k,n}}=\O(2)\to T_{\gG}|_{C_{k,n}}$ corresponds to te inclusion
	$\kk1_V\to V^*\otimes V$ and consequently
	\[N_{X/\gG}=(V^*\otimes V/\kk1_v)\otimes\O(2)\]
	which is evidently perfectly balanced of slope 2. 
	Thus given 3 $\P^k$s in $\P^{2k+1}$, there are finitely many
	minimal scrolls containing them as rulings.\par
	Now assume $n\geq 2k+2$.  Write 
	\[n+1=q(k+1)+r, q\geq 2, 0\leq r<k+1.\] 
Because $N_{X_{k,n}/\gG}$ has degree $(n-k)(n+1)-2$ and rank $(n-k)(k+1)-1$,
	an elementary calculation shows that $q$ is the slope of $N_{X_{k,n}/\gG}$.
	We need to show that given general
	$L_1,...,L_{q+1}\in \gG$ there is a minimal scroll containing them as rulings,
	or equivalently, a $C_{n-k}\subset\gG$ through $[L_1],...,[L_{q+1}]$.
	Let $A=\carets{L_1,...,L_{q-1}}$ and let $B$ be a general 
	 $\P^{n-k-1}$ containing $A$. Let $L_0$ be a general $\P^k$
	 contained in $B\cap\carets{L_q, L_{q+1}}$. By the $n=2k+1$ case above, there exists a 
		$C_{k+1}$ in $\gG(1,\carets{L_q, L_{q+1}})$ through  $L_0, L_q, L_{q+1}$. 
		By induction, there is a minimal scroll in $B$ containing $L_0, L_1,...,L_{q-1}$
		as rulings, corresponding to a minimal curve $C_{n-k-1}\subset\gG(k, \P^{n-k-1})$.
		Then $C_{n-2k-1}\cup C_{k+1}$ smooths out to a $C_{n-k}\subset\gG$ 
		through $q+1$ general points. Therefore $C_{k,n}$ is balanced.
		
	\end{proof}
\begin{cor}\label{remainder}
	Let $R=R(k,n)$ be the remainder of $(n-k)(n+1)-2$ divided by $(n-k)(k+1)-1, k<n/2$.
	Let $\Sigma\subset\gG$ be a sufficiently general Schubert cycle
	of codimension $A=R+1$ meeting $C_{k,n}$
	at a point. Then the birational transform of $C_{k,n}$ in $B_\Sigma\gG$
	is perfect.
	\end{cor}
\begin{proof}
	If $L\in\gG$ is a general point of $C_{k,n}$, 
	the upper subspace of $N_{C_{k,n}}$ at $L$ corresponds to
	some $(k+R)$-dimensional subspace $M\subset\P^n$ containing $L$.
	We choose the standard flag $\P^0\subset...\subset\P^n$
	general in general position with respect to $L, M$,
	 subject to $L\in\Sigma(a.)$. Then $T_L\Sigma(a.$ and $M$ will be in general position
	modulo $L$. This makes the birational transform of $C_{k,n}$ perfect.
	\end{proof}
\begin{rem}
	Write
	\[n+1=q(k+1)+r, 0\leq r\leq k.\]
	We have $q\geq 2$ as $n\geq 2k+1$.\par
By an elementary calculation, we have
\eqspl{R(n,k)}{R=q-2+r(n-k).\qed}
\end{rem}

We now present some examples of perfect curves as birational transforms of $C=C_{k,n}$ in a blowup of $\gG$
in a suitable Schubert cycle $\Sigma(a.)$, as in Corollary \ref{remainder}.

\begin{example}\label{r=0} Assume $r=0$, i.e. $n+1=q(k+1)$. Then $R(k,n)=q-2$,
	so the upper subbundle of $N_{C_{k,n}/\gG}$ has rank $q-2$.
	If $q=2$ then $N_{C_{k,n}/\gG}$ is already perfect, so assume $q>2$.
	Then let $(a.)=(q-1, 0,...)$.
We take for $\Sigma$ the Schubert cycle
	\eqspl{sigma1}{\Sigma=\Sigma(a.)=\{L\in\gG: L\cap \P^{n-k-(q-1)}\neq\emptyset\}}
	which has codimension $q-1$ in $\gG$. Then as in Corollary
	\ref{remainder}, the birational transform $\tilde C$ in $B_\Sigma \gG$ of $C$ meeting $\Sigma$ 
	is perfect.\end{example}
\begin{example}\label{alphabeta-example}
	Assume we can write
	\eqspl{alphabeta}
	{ R+1= \alpha(n-k)+\beta(k+1)-\alpha\beta.
	}
	Then the Schubert cycle
	\[\Sigma=\gG(k-\alpha, n-\alpha-\beta)=\Sigma((n-k)^{(\alpha)}, \beta^{(k+1-\alpha)})\]
	has codimension $R+1$ so the birational transform of $C_{k,n}$ meeting $\Sigma$
	in $B_\Sigma\gG$ is perfect.\par
	A sufficient condition to have the above form for $R+1$ is 
	\[(k+1-r)|(q-1).\] Then we can take
	\[\alpha=r, \beta=\frac{q-1}{k+1-r}.\]
	A sufficient condition for \eqref{alphabeta} to hold with $\alpha=0$, i.e. $R+1=\beta(k+1)$
	is 
	\[q-1+r^2\equiv 0\mod k+1.\]\par
	Note that the normal bundle $N=N_{\Sigma/\gG}$ admits a filtration with graded pieces\nl
	$\alpha Q_{n-k-\beta}, \alpha\beta\O, \beta S^*_{k+1-\alpha}$.
	\end{example}
%
%

/******
%
************/
\subsection{Balanced rational curves in strict flag resolutions and Schubert supercycles}\label{c+}
Here we describe a construction for rational curves that will be applied below to construct balanced
rational curves in Grassmannians.
\subsubsection{Starting curve} In this section we think of Grassmannians as prametrizing vector subspaces. 
Notations as in \S \ref{schubert} with $A=\sum a_i=R+1$ we
start with a curve 
\[\P^1\to C_0\subset \Sigma_0=G(k_0+1, n-k-a_{k_0}+k_0+1)    \]
of degree $e_0$ (NB we allow the case where $\Sigma_0$ is a point,
i.e. $a_{k_0}=n-k$, in which case $e_0=0$, but in case $a_0<n-k$ we require
$C_0\to\Sigma_0$ to be an immersion so either $\Sigma_0\neq\P^1$ or $e_0=1$ ). 
Then $Q_0$ has degree $e_0$ on $C_0$, hence $C_0$ has slope 
\[s_0=(e_0(n-k-a_{k_0}+k_0+1)-2)/((n-k-a_{k_0})(k_0+1)-1).\]
If $\Sigma_0$ is a point, $s_0$ is deemed undefined.
By induction, these are infinitely many choices of $e_0$ (even all large enough $e_0$
if $k_0=0$), such that
 a general $C_0$ is balanced.\par
 \subsubsection{Lifting within $\tilde\Sigma$}
A general lifting of $C_0$ to $C_1\to\Sigma_1$ amounts to a general surjection over $C_0$,
\[F_1^*|_{C_0}\to S^{1*}_{k_1-k_0}|_{C_0}\] where $S^{1*}_{k_1-k_0}|_{C_0}$ is a given
bundle of degree $e_1$ which we assume perfect, so in particular 
\[(k_1-k_0)|e_1, S^{1*}_{k_1-k_0}|_{C_0}=(k_1-k_0)\O(e_1/(k_1-k_0)).\]
Assume first $C_0$ is an immersed curve.
Then the corresponding kernel $Q^1_{n-k-a_{k_1}}$
will be balanced and so the vertical normal bundle of $C_1$, i.e. $S^{1*}\otimes Q^1$
 will be balanced of slope
\eqspl{s1-slope}{s_1=e_1/(k_1-k_0)+(e_1+e_0)/(n-k-a_{k_1}).}
If $C_0$ is a point, the 'vertical normal bundle' is just the normal bundle
whose slope is
\eqspl{s1-slope0}{
s_1=\frac{e_0(n-k+k_1-k_0-a_{k_1})-2}{(k_1-k_0)(n-k-a_0)-1}.	
}
Note that even if $C_0$ is a point, $C_1$ will be an immersed curve.
Similarly in the general case  general lifting of $C_i$, assumed immersed,
 to $C_{i+1}\to\Sigma_{i+1}$ 
amounts to a general surjection over $C_i$,
$F_{i+1}^*|_{C_i}\to S^{(i+1)*}_{k_{i+1}-k_i}$ where $S^{(i+1)*}_{k_{i+1}-k_i}|_{C_i}$ is a given
bundle of degree $e_{i+1}$ which we assume perfect, so in particular 
\[(k_{i+1}-k_i)|e_{i+1}\] hence
\[S^{i+1}_{k_{i+1}-k_i}=(k_{i+1}-k_i)\O(e_{i+1}/(k_{i+1}-k_i)).\]
 Then the corresponding kernel $Q^{i+1}_{n-k+1-a_{k_{i+1}}}$
will be balanced and so the vertical normal bundle of $C_1$, i.e. $S^{(i+1)*}\otimes Q^{i+1}$,
will be balanced of slope
\[s_{i+1}=e_{i+1}/(k_{i+1}-k_i)+(e_{i+1}+e_i+...+e_0)/(n-k-a_{k_{i+1}})\]
(slope of tensor product=sum of slopes).
\par
\subsubsection{Lifting to $\Sigma^+$}
Now  the normal bundle $N=N_{\tilde\Sigma_*(a.)\to\gG}$ as described in \S \ref{normal-sec}
is filtered with quotients $S^{i*}\otimes G_i$ of degree $e_{i+1}+...+e_u+a_ie_i$
which are balanced of slope
\[g_i=e_i/(k_i-k_{i-1})+(e_{i+1}+...+e_u
)/((k_{i}-k_{i-1})a_i).\]
A general lifting of $C_r$ to $C_+\to\tilde\Sigma^+=\P(N\oplus\O)$ amounts to a general surjection over $C_s$,
$N^*\oplus\O\to\O(e_+)$. Now if $N$ is balanced then the kernel is always balanced.
In general,  Corollary \ref{kernel} yields that the kernel will be balanced provided 
\eqspl{e+}{e_+\geq \deg(N)-(\rk(N)-1)\min(g_i)=(\sum\limits_{i=0}^uie_{i}+a_ie_i-(A-1)\min(g_i).
}
Then the vertical normal bundle of $C_+$ will be balanced of slope
\eqspl{s+}
{s_+=\deg(N)/A+e_+(A+1)/A=(\sum\limits_{i=0}^uie_{i}+a_ie_i+(A+1)e_+)/A.
}

Finally $C_+\to\P(N\oplus\O)$ will be balanced provided
\eqspl{final-slope}{[s_0]=...=[s_u]=[s_+].}
In case $\Sigma_0$ is a point, $[s_0]$ is undefined and omitted from te above equation.
Note that in that case $S^0_{k_0+1}$ is trivial and $Q^0_{n-k}$ has degree $e_0+...+e_u$
(where $e_0=0$).
Note that by generality $C_+$ is disjoint from the exceptional locus of $\tilde\Sigma*(a.)\to\Sigma(a.)$
so the normal bundle to $C_+$ in $\tilde\Sigma*(a.)$ coincides with that in 
$\Sigma^+(a.)=\P(N_{\Sigma(a.)/\gG}\oplus\O$.

\begin{example}
	The results of \S\S \ref{odd-sec}, \ref{4n_1+2} for $k=1$ can be recovered
	by using a suitable Schubert (=strict Schubert)
	flag resolution as above:\par
	For $n=2n_0+1$ odd (see \S \ref{odd-sec}), we have $R=(n-3)/2, A=(n-1)/2=n_0$.
We can use $(a)=(a_0)=(n_0)$.\par
	
	For $n=4n_1+2$ (see \S \ref{4n_1+2}), we have $A=R+1=6n_1+1$
	We can use  $a=(3n_1+1, 3n_1)$.
	
\par
	For $n=4n_1$ the  (see \S \ref{4n_1}) we must use a strict flag resolution. We have 
	$R=3n/2-3, A=R+1=6n_1-2$ and we can use $(a.)=(n_1+1)$. We obtain the same results as before
	with $\Sigma_*(a.)$.
	An alternative choice for $(a.)$
	here is $a_0=4n_1-3, a_1=2n_1+1$. Then $e_0$ is arbitrary subject
	only to $2e_0\not\equiv n_1\mod n-2$ and $e_0$ then determines $e_1$.
	\par
	\end{example}
/***********
*************/
\subsection{Balanced scrolls}
Now to construct a balanced curve in $\gG$ we use two dual strategies.
Both strategies use the standard degeneration
of $\gG$ as in \S \ref{super}, blowing up  $\gG\times\A^1$ in $\Sigma\times 0$,
which yields a degeneration of $\gG\times 1$
to the inverse image of $\gG\times 0$ which is
\[X_0=\fF_1\cup_E\fF_2\] where $\fF_1=B_\Sigma\gG$, $\fF_2=\P(N\oplus\O)=\Sigma^+,
N=N_{\Sigma\to\gG}$ and $E=\P(N)$.  The two strategies differ in the construction of
a good connected nodal curve
$C_1\cup C_2\subset X_0$.\par
The first strategy, which might be called 'perfect plus balanced' (p+b) is used in Theorem \ref{any k}. 
It proceeds just as in the case $k=1$ (see \S 
\ref{degeneration}). We choose $A=R(k,n)+1, \Sigma=\Sigma(a.)$ with 
$\sum a_i=\sum (k_i-k_{i-1})a_{k_i}=A$, and 
$C_1=\bigcup C_{1,i}$ with $C_{1,i}$ as in Corollary \ref{remainder} the birational
transform of a minimal curve $C_{k,n}$, hence perfect.
For  $C_2\subset \fF_2$ we take a balanced curve,
corresponding to a balanced curve  $C_+\subset \tilde\Sigma^+(a.)$ (they have the same
normal bundle),
with $C_2\cap E=C_1\cap E$. Then $C_1\cup C_2$ will smooth out to a balanced curve in
$\gG$ of degree 
\[e=\deg(C_1)+\deg(C_2)-\deg(C_i\cap E).\]\par
Using the notation of \S \ref{c+}, we have
\eqspl{final-e}{
		e=e_0+...+e_u+(n-k-1)e_+
}
The second strategy, which might be called 'balanced plus perfect' (b+p)is used in Theorem \ref{mirror}.
It is a 'mirror' of the first, making $C_1\subset\fF_1$ perfect
and irreducible with $e_+=1$ and $C_2$ just balanced of degree $d$, the proper transform
of a balanced curve $C_2'$ in $\gG$ (initially, $d=n-k$ so $C_2'$ is minimal). This
yields a balanced curve of degree $e=d+\sum\limits_{i=0}^ue_i$. Repeating the process,
we get balanced curves of degree
\[e=n-k+r\sum\limits_{i=0}^ue_i, r\geq 0.\]\par 
\begin{thm}[p+b]\label{any k}
	Notations as above with $2k<n$, assume \eqref{final-slope} holds and either\par
	(i) $k_0=0$, or\par
	(i)' $\gG(k_0, n-k-a_{k_0}+k_0)$ contains a balanced curve of degree $e_0$\par
	and either\par
	(ii) \eqref{e+} holds, or\par
	(ii)' $u=0$.\par
	Then $\gG(k,n)$ contains a balanced rational curve of degree
	$e=e_0+...+e_u+(n-k-1)e_+$.
	\end{thm}
\begin{example}[Example \ref{r=0} cont'd] Let $n+1=q(k+1)$. Then Equation \eqref{s+}
	becomes, with $m:=n-k-q+1$:
	\eqspl{c+-balanced}{
		e_0+e_++[e_+/(k+q)]=[(2ke_1-e_0)/n-k-1)]=[e_0(m+1)/(m-1)].
	}
Now assume $(2k, q-1)=1$, hence $(2k, n-k-1)=1$ and let $j2k\equiv 1 \mod n-k-1$. 
Then note that by choosing $e_0\equiv 0\mod m-1, 
e_1\equiv je_0\mod n-k-1 $ we get an arithmetic progression of values of 
$e=e_0+e_1+e_+(n-k-1)$. For example, for $k=2$ the condition $(2k, q-1)=1$
means $n+1$ is divisible by 6.

\end{example}
\begin{example}[Example \ref{alphabeta-example} cont'd]\label{ab+}
	Assume \eqref{alphabeta} holds. It is then convenient to avoid the general formalism
	of this section and just start with $C_0\subset\Sigma_0=\gG(k-\alpha, n-\alpha-\beta)$
	balanced of degree $e_0$ and slope $s_0$ 
	and try to lift it to $C_+$. We need $N_{C_+}$ to be balanced for a general lifting, which is automatic
	if either $\alpha=0$ or $\beta=0$.  If
	$\alpha, \beta>0$ then $N_{C_+}$ for a genera lifting will be balanced provided
	\eqref{e+} holds, which here reads 
	reads
	\eqspl{e+alpha beta}
	{e_+\geq\alpha e_0(\frac{n-k}{k+1-\alpha}-1).
	}

Then the slope matching condition becomes
	 \eqspl{s+matching}{[s_0]=[s_+].} 
Note that here we have
\[s_+=e_++\frac{e_++e_0(\alpha+\beta)}{\alpha(n-k)+\beta(k+1-\alpha)}, 
s_0=\frac{e_0(n+1-\alpha-\beta)-2}{(k+1-\alpha)(n-k-\beta)-1}.\]
Case (i): assume $\alpha=0, \beta<<k$.\par  Then solving \eqref{s+matching} for $e_+/e_0$
yields asymptotically
\[\frac{n+1-\beta}{(k+1)(n-k-\beta)}-1/k \sim n/k(n-k)-1/k \sim 1.\]
Since this is positive, \eqref{s+matching} has infinitely many solutions $e_0, e_+$
for large, fixed $n, k$.\par
Case (ii): assume $\beta=0, \alpha<<k$.\par Again \eqref{s+matching} yields asymptotically
\[e_+/e_0 \sim \frac{n+1-\alpha}{(k+1)(n-k)}-1/(n-k) \sim \frac{n-k-\alpha}{(k+1)((n-k)}>0.\]
hence again there are infinitely many solutions.
\par
Case (iii) assume $q=2, \alpha=1, \beta=n-k-\gamma, \gamma<<k$ and also
$r/k\leq c<1$ as $k$ gets large, for a constant $c$.\par
 Thus $n=2k+1+r$.
Then \eqref{e+alpha beta} says 
\eqspl{case3}{e_+/e_0\geq (r+1)/k}
where $(r+1)/k\leq c+1/k<1$,
 while \eqref{s+matching} says asymptotically
\eqspl{case3'}
{e_+/e_0 \sim (k+\gamma)/k\gamma - \frac{n-k+1-\gamma}{(n-k)(k+1)-\gamma(k-1)}\sim 1-1/k>>(r+1)/k}
(the last inequality holds as soon as $k>2/(1-c)$). So \eqref{case3} follows asymptotically from 
\eqref{case3'}.
Thus we get infinitely many solutions in this case as well.

	\end{example}
We now turn to  the mirror b+p strategy to that of Theorem \ref{any k}, as mentioned above.
Here we make $C_+$ and $C_2$ perfect
with $e_+=1$.
Then attach $\bigcup\limits_{i=1}^mC_{2, i}$ to 
a balanced, not necessarily perfect curve $C_1\subset\fF_1$  with $C_1.E=m$. 
In case $m=1$, we can take for $C_1$ the birational transform of
a balanced curve $C'_1$ in $\gG$ (e.g. a minimal curve) meeting the Schubert cycle $\Sigma$ 
transversely in 1 point. When the codimension of $\Sigma$ does not match the upper rank of $C'_1$,
i.e. \eqref{alphabeta} fails, 
the proper transform $C_1$ will still be balanced though not perfect (see \cite{caudatenormal}, Lemma 4).
This yields the following.
\begin{thm}[b+p]\label{mirror}
	Hypotheses as in Theorem \ref{any k}, assume moreover that 
	\eqspl{integer}{s_0,...,s_u\in\Z,
	e_+=1,} and that
	$C_1\subset \fF_1$ is balanced of degree $d$ with $C_1.E=m$. Then $\gG(k,n)$ contains a balanced
	rational curve of degree $e=d+m(e_0+...+e_u)$.
	\end{thm}
\begin{cor}\label{oddG}
	$\gG(k, 2k+1), k\geq 1$ contains balanced rational curves of degree $rk+1$ for all $r\geq 1$. 
	\end{cor}
\begin{proof}
	Working as in Example \ref{ab+}, take $n=2k+1, \alpha=\beta=1, 
	\Sigma=\gG(k-1, 2k-1), e_0=k$. Thus $C_0\subset \gG(k-1, 2k-1)$
	is minimal hence balanced. 
	In this case it is easy to check directly that $N|_{C_0}=2k\O(1)+\O$
	which is balanced, so the kernel of a general map $N^*|_{C_0}\oplus\O\to\O(1)$ is just
	$n\O(-1)$. Thus $s_+=s_0=2$ so $C_+$ is perfect.
	For $C_1$ we take the birational transform of a minimal curve
	in $\gG$ and $m=1$.
	Then the result follows from the Theorem.
	\end{proof}
For example, in $\P^3$ (resp. $\P^5$), there are balanced surface (resp. 3-fold) scrolls
of every degree (resp. every odd degree) $>1$, etc.
\begin{example}
	Let \[\Sigma=\Sigma((\beta-1)^{(k)})=\{ L: \dim(L\cap\P^{n-\beta})\geq k-1\}\subset \gG(k,n).\]
	Thus, 
	\[\Sigma_0=\gG(k-1, n-\beta), \Sigma_1=\P_{\Sigma_0}(Q_{n-k-\beta+1}\oplus\beta\O),
	\Sigma^+=\P_{\Sigma_1}(N\oplus\O).\]
	We have
	\[N=S_{\Sigma_0}^*\otimes Q_{\Sigma_1}\]
	where $Q_{\Sigma_1}$ is locally free on the big open set where $\dim(L\cap\P^{n-\beta})=k-1$,
	and there fits in an exact sequence
	\[\exseq{Q^*_{\Sigma_1}}{\beta\O}{\O_{\Sigma_1}(1)}.\]
	Working as in Theorem \ref{mirror}, we get
	\[s_0=\frac{e_0(n+1-\beta)-2}{k(n-\beta-k+1)-1}, s_1=\frac{e_0+e_1}{n-k}+e_1,
	s_+=\frac{e_0}{k}+\frac{e_1}{(\beta-1)}+\frac{1}{k(\beta-1)}+1.\]
	We must solve
	\[s_0=s_1=s_+\in\Z.\]
	A solution yields balanced rational curves in $\gG(k, n)$ of degree 
	\[e=n-k+r(e_0+e_1), r\geq 1.\]
	These equations seem to be difficult to solve. A couple of isolated solutions are:\par
	(i) $n=8, k=3, \beta=3, e_0=7, e_1=3$. This yields balanced curves in $\gG(3, 8)$ of degree
	$5+10r, r\geq 0$.\par
	(ii) $n=5, k=1, \beta=3, e_0=3, e_1=5$, yielding balanced curves
	in $\gG(1, 5)$ of degree $4+8r, r\geq 0$.
	\end{example}

/****************
***********/

\bibliographystyle{amsplain}
\bibliography{../mybib}

\providecommand{\bysame}{\leavevmode\hbox to3em{\hrulefill}\thinspace}
\providecommand{\MR}{\relax\ifhmode\unskip\space\fi MR }
\providecommand{\MRhref}[2]{%
  \href{http://www.ams.org/mathscinet-getitem?mr=#1}{#2}
}
\providecommand{\href}[2]{#2}
\begin{thebibliography}{1}

\bibitem{chen-zhu}
Q.~Chen and Y.~Zhu, \emph{Very free curves on {F}ano complete intersections},
  Algebr. Geom. \textbf{1} (2014), no.~5, 558--572.

\bibitem{coskun-riedl}
I.~Coskun and E.~Riedl, \emph{Normal bundles of rational curves on complete
  intersections}, arXiv:math.AG (2017), 1705.08441v1.

\bibitem{eisenbud-harris-minimal}
D.~Eisenbud and J.~Harris, \emph{On varieties of minimal degree}, Algebraic
  Geometry Bowdoin 1985 (Spencer~J. Bloch, ed.), vol.~46, Proc. Symp. Pure
  Math., no.~1, Amer. Math. Soc, 1987, pp.~3--13.

\bibitem{caudatenormal}
Z.~Ran, \emph{Balanced curves and minimal rational connectedness on {F}ano
  hypersurfaces}, arxiv.math:2008.01235.

\bibitem{ran-normal}
\bysame, \emph{Normal bundles of rational curves in projective spaces}, Asian
  J. Math. \textbf{11} (2007), 567--608.

\bibitem{hypersurf}
\bysame, \emph{Low-degree rational curves on hypersurfaces in projective spaces
  and their fan degenerations}, J. Pure Applied Algebra (2020), Arxiv
  1906.03747.

\bibitem{shen-normal}
M.~Shen, \emph{Normal bundles of rational curves on {F}ano 3-folds}, Asian J.
  Math. \textbf{16} (2012), 237--270.

\end{thebibliography}

\end{document}